\documentclass[]{interact}
\usepackage{latexsym}
\usepackage{anyfontsize} 
\usepackage{tikz} 
\usetikzlibrary{cd} 
\usepackage{color}
\usepackage{hyperref}
\usepackage{url}
\usepackage{breakurl}
\newcommand{\bburl}[1]{\textcolor{blue}{\url{#1}}}

\makeatletter
\newcommand{\monthyear}[1]{%
  \def\@monthyear{\uppercase{#1}}}
\newcommand{\volnumber}[1]{%
  \def\@volnumber{\uppercase{#1}}}
\makeatother

\theoremstyle{plain}
\numberwithin{equation}{section} 
\newtheorem{thm}{Theorem}[section] 
\newtheorem{theorem}[thm]{Theorem}
\newtheorem{lemma}[thm]{Lemma}
\newtheorem{example}[thm]{Example}

\newtheorem{conjecture}[thm]{Conjecture}

\newtheoremstyle{italicremark} 
  {3pt}   
  {3pt}   
  {\itshape}  
  {}      
  {\bfseries} 
  {.}     
  { }     
  {}      

\theoremstyle{italicremark}
\newtheorem{remark}[thm]{Remark}

\numberwithin{table}{section} 
\numberwithin{figure}{section}

\usepackage{tikz,lipsum,lmodern}
\usepackage[most]{tcolorbox}

\begin{document}

\monthyear{Month Year}
\volnumber{Volume, Number}
\setcounter{page}{1}

\title{Fibonacci Numbers and Vieta Jumping for a Rational Diophantine Equation}

\author{
\name{Steven J. Miller\textsuperscript{a}, Dimitrios Nikolakopoulos\textsuperscript{b} and Anitha Srinivasan\textsuperscript{c}\thanks{Email address: \textsuperscript{a}sjm1@williams.edu, \textsuperscript{b}bko25001@uconn.edu and  \textsuperscript{c}rsrinivasananitha@gmail.com}}
\affil{\textsuperscript{a}Department of Mathematics,
                Williams College,
                Williamstown, MA 01267, USA;
        \textsuperscript{b}Department of Mathematics,
                Univerity of Connecticut,
                Storrs, CT 06269, USA;
                \\
      	\textsuperscript{c}Department of Mathematics,
                Universidad Pontificia de Comillas, ICADE, Alberto Aguilera, 23, 28015 Madrid, Spain}
}

\maketitle

{\bf Article type}: research 
\bigskip

\begin{abstract}
We study the Diophantine equation $\displaystyle{\tfrac{a+1}{b} + \tfrac{b+1}{a} \ = \ k}$, where $k$ is an integer.
Using Vieta jumping, we completely classify all positive integer pairs $(a, \, b)$.
We prove that the associated integer value $k$ can only be $3$ or $4$.
The corresponding solution pairs $(a,\,b)$ are related to the classical Fibonacci numbers.
As a consequence, the quantity $\frac{a+b}{\gcd(a, \,b)^2}$ takes only the values $1, \, 2, \, 3$ and $5$.
This reveals an unexpected connection between a simple rational Diophantine condition, Vieta jumping, and Fibonacci numbers.
\end{abstract}




\section{Introduction}
We study positive integer solutions $(a, \,b)$ of
\begin{equation}\label{eq:main_equation}
\frac{a+1}{b} \, + \, \frac{b+1}{a} \ = \ k,
\qquad a, \, b \in \mathbb{Z}_{>0},
\end{equation}
where $k$ is an integer. Related Diophantine equations have been studied in \cite{lemmermeyer2026vietajumpingsmallnorms}.
It is known from the OEIS~\cite{OEIS} that all positive integer solutions for $k \ = \ 3$ and $k \ = \ 4$ are described by the sequences A032908 and A101879, respectively.
More precisely, consecutive terms of A032908 give all solutions for $k \ = \ 3$, while consecutive terms of A101879 give all solutions for $k \ = \ 4$.

Using the Fibonacci sequence
\begin{equation}
F_{-1} \ = \ 1, \qquad
F_{0} \ = \ 0, \qquad
F_{1} \ = \ 1, \qquad
F_{n} \ = \ F_{n-1} \,+\, F_{n-2}
\quad (n \geqslant 2),
\end{equation}
together with Vieta's formulas: if $x_{1}, \, x_{2}$ are the roots of the quadratic polynomial $x^{2} \,-\, Sx \,+\, P,$ then $x_{1} \,+\, x_{2} \ = \ S$ and $x_{1}x_{2} \ = \ P$,
we prove that $k$ can take only the values $3$ and $4$, and we determine the complete structure of the solutions of~\eqref{eq:main_equation}.
The Fibonacci description of the solutions for $k \ = \ 3$ is also mentioned in the OEIS~\cite{OEIS}, where A032908 is identified with the sequence $\{F_{2n-1}+1\}_{n\geqslant 0}$. However, we are not aware of a complete proof in the literature.
We also study the associated greatest common divisor structure of the solutions, which appears to be new to the best of our knowledge.

A central tool in our analysis is the method of \emph{Vieta jumping}, which constructs new integer solutions of a quadratic Diophantine equation from existing ones by replacing one root with the other root determined by Vieta's formulas.
Starting from a positive integer solution, one ``jumps'' to another solution sharing one coordinate with the original pair, hence the name \emph{Vieta jumping}. For further background on Vieta jumping, see \cite{engel_problem_solving, GEY, stevens_olympiad_nt, wu_vieta2024}.

\medskip

We first define two operations to solutions $(a, \, b)$ to \eqref{eq:main_equation} that generate new solutions: flipping and Vieta jumping; it is trivial to see these are solutions by direct substitution.
\begin{itemize}\itemsep0.5em
\item Flipping: The flip of $(a, \, b)$ is the pair $(b, \, a)$, and is also a solution.

\item Vieta Jumping: The Vieta jumps of $(a, \, b)$ are $(kb-1-a, \, b)$ and $(a, \, ka-1-b)$, and both are also solutions.
\end{itemize}

\begin{theorem} \label{thm:main}
Let $a, \, b\in\mathbb{Z}_{>0}$ be such that
\begin{equation} \label{equationTHM1.1}
\frac{a+1}{b} \; + \; \frac{b+1}{a} \ = \ k
\end{equation}
where $k$ is an integer.
Then $k$ is $3$ or $4$.
Any solution $(a, \, b)$ with distinct coordinates comes from a solution with minimal coordinate sum through a finite sequence of Vieta jumps and flips.
Conversely, any such sequence starting from a solution with minimal coordinate sum produces solutions with distinct coordinates.

\begin{enumerate}
\item \label{propert(i)} If $k \ = \ 3$: the solution with the minimal coordinate sum is $(2, \, 2)$, all other solutions are generated by repeated applications of Vieta jumps and flips. Moreover, every solution $(a, \, b)$ with $a \, \leqslant \, b$  is of the form
\begin{equation}
a \ = \ F_{2n-1}\,+\,1 , \quad b \ = \ F_{2n+1}\,+\,1 \quad \text{ for some } n \geqslant0,
\end{equation}
where $\big \{ F_{n} \big \}_{n \geqslant -1}$ denotes the Fibonacci sequence.

\item \label{propert(ii)} If $k \ = \ 4$: the solution with the minimal coordinate sum is $(1,\,1)$, every solution is obtained from $(1,\,1)$ by repeated application of Vieta jumps and flips.
\end{enumerate}
\end{theorem}

Combining denominators in \eqref{eq:main_equation} we obtain $a^{2} \,+\, a \,+\, b^{2} \,+\, b \ = \ k a b$. Since $\gcd(a, \, b)^{2}$ divides $a^{2}\,+\,b^{2}\,+\,a\,+\,b$, as the sum $k$ is an integer, thus $\gcd(a, \, b)^{2}$ divides $a\,+\,b$ as it divides $a^{2}$ and $b^{2}$.
Thus, for the solutions $(a, \, b)$ described in Theorem~\ref{thm:main}, the quantity
\begin{equation}\label{eq:normalized_sum}
\frac{a+b}{\gcd(a,\,b)^2}
\end{equation}
is an integer. We prove it takes only a small set of values.

\begin{theorem}\label{cor:normalized_values}
For every solution $(a,\,b)$ of \eqref{eq:main_equation}, we have
\begin{equation}
\frac{a+b}{\gcd(a, \,b)^2} \, \in \, \{ 1, \, 2, \, 3, \, 5 \},
\end{equation}
the first four Fibonacci numbers.
More precisely, the values $1$ and $5$ occur when $k \ = \ 3$, while the values $2$ and $3$ occur when $k \ = \ 4$.
\end{theorem}

The paper organized as follows.
In Section~\ref{sec:two_var_general_r} we prove Theorem~\ref{thm:main}. The proof is divided into two main parts. First, we classify the solutions with equal coordinates. 
Second, we show
that every solution with distinct coordinates can be reduced, via Vieta jumping and flipping, to another solution with strictly smaller coordinate sum. Iterating this process eventually leads to a solution with equal coordinates. We then prove that parts~\ref{propert(i)} and~\ref{propert(ii)}.
In Section~\ref{section: lasttheorem} we prove Theorem~\ref{cor:normalized_values}, separating the cases $k \ = \ 3$ and $k \ = \ 4$.
Finally, in Section~\ref{sec:limitations_open} we discuss limitations of the Vieta jumping method in further generalizations and related analogues, where integrality may fail, solutions with equal coordinates may not exist, and solutions may split into disjoint families not connected by Vieta jumps.

In addition to uncovering relationships that are of interest in their own right, a major purpose of this project is to mentor young mathematicians, showing them the excitement of discovering new results and guiding them in how to conduct original research (from gathering data to making conjectures to writing proofs). Not surprisingly, these problems are thus also very amenable to analysis by AI; we ask readers interested in exploring these questions to keep the above in mind.

\section{Proof of Theorem \ref{thm:main}} \label{sec:two_var_general_r}

We define, for each positive integer $k$, the solution set
\begin{equation}
S_k \ := \ \left\{ (a,\,b)\in\mathbb{N}\times\mathbb{N} \; : \; \frac{a+1}{b} \ + \ \frac{b+1}{a} \ = \ k \right\}.
\end{equation}
Our goal is to determine all integers $k$ for which $S_k \neq \varnothing$ and to describe all positive integer solutions $(a, \,b)$ of~\eqref{eq:main_equation}.
We show that non-empty solution sets $S_{k}$ occur only for $k \ = \ 3$ and $k \ = \ 4$.

We begin by showing that $S_{k} \ = \ \varnothing$ for all $k \, \leqslant \, 2$.
Suppose that $(a, \, b) \in S_k$. Let $x \ = \ a/b \,>\, 0$. Then
\begin{equation}
k
\ = \
\frac{a+1}{b} \, + \, \frac{b+1}{a}
\ = \
\left( \frac{a}{b} \,+\, \frac{b}{a} \right)
\,+\,
\frac{1}{a} \,+\, \frac{1}{b}
\ = \
x \,+\, \frac{1}{x}
\,+\,
\frac{1}{a} \,+\, \frac{1}{b}.
\end{equation}
By the arithmetic--geometric mean inequality, $\tfrac{x \,+\, 1/x}{2} \geqslant \sqrt{x \cdot \tfrac{1}{x}} \ = \ 1$,
and hence $x\,+\, 1/x \geqslant 2$.
Therefore, $k \, > \, 2$ since $a, \, b \, > \, 0$.

First, we determine whether $S_{k}$ contains a solution of the form $(u, \, u)$. Such solutions are called \emph{diagonal solutions}; all other solutions are called \emph{non-diagonal solutions}.

\begin{lemma} \label{Lemma: diagonalIMPLIESk34}
The only diagonal solutions to~\eqref{eq:main_equation} are $(2,\,2)$ (corresponding to $k\ = \ 3$)
and $(1,\,1)$ (corresponding to $k\ = \ 4$).
\end{lemma}

\begin{proof}
Let $(u, \, u) \in S_{k}$ be a diagonal solution.
Then $(u+1)/u \,+\, (u+1)/u \ = \ k$,
so $(2u+2)/u \ = \ k$,
thus $2\,+\,2/u \ = \ k$.
Hence $u$ divides $2$, and therefore $u \ = \ 1$ gives $k \ = \ 4$,
while $u \ = \ 2$ gives $k \ = \ 3$.
We conclude that the only diagonal solutions are $(1, \, 1)$ for $k \ = \ 4$ and $(2, \, 2)$ for $k \ = \ 3$.
\end{proof}

The following lemma shows that these solutions have a minimal coordinate sum for their corresponding values of $k$.

\begin{lemma} \label{lemma: minimalcoordinatesum}
If $k$ is $3$ or $4$, then the solutions with minimal coordinate sum are $(2,\,2)$ and $(1,\,1)$, corresponding to $k \ = \ 3$ and $k \ = \ 4$, respectively.
\end{lemma}

\begin{proof}
For $k \ = \ 3$, the points $(1,\,1)$, $(1,\,2)$, and $(2,\,1)$ are not solutions. Hence every solution in $S_3$ has coordinate sum at least $4$, and equality is attained by $(2,\,2)$.
\\
For $k \ = \ 4$, the diagonal solution $(1,\,1)$ has coordinate sum $2$, which is minimal among all positive integer solutions in $S_{4}$.
\end{proof}

The next theorem shows that every non-diagonal solution can be reduced to another solution with a strictly smaller coordinate sum using Vieta’s formulas. Iterating this process eventually leads to a diagonal solution.

\begin{theorem} \label{Lemma: smallerCOORDINATEsum}
Let $(a,\,b) \in S_k$ be a non-diagonal solution and without loss of generality we may assume $a \, > \, b$.
Then there exists a positive integer $a'$ such that
\begin{equation*}
(a',\,b)\in S_k,
\qquad
a'\,<\,a,
\qquad\text{and}\qquad
a'\,+\,b\,<\,a\,+\,b.
\end{equation*}
Consequently, any non-diagonal solution has a chain of solutions ending in a diagonal solution, which has a minimal coordinate sum.
\end{theorem}

\begin{proof}
A solution pair $(a, \, b) \in S_{k}$ satisfies
\begin{equation} \label{1k}
\frac{a + 1}{b} \, + \, \frac{b + 1}{a} \ = \ k,
\end{equation}
or equivalently
\begin{equation}\label{eq:quad_in_a}
a^{2} \,+\, a \,+\, b^{2} \,+\, b \ =\ k a b.
\end{equation}
Viewing~\eqref{eq:quad_in_a} as a quadratic equation in $a$, we obtain
\begin{equation}
a^2 \, - \, (kb-1)a \, + \, (b^2+b) \ = \ 0.
\end{equation}
Hence the polynomial
\begin{equation}\label{eq:quad}
x^2 \, - \, (kb-1)x \, + \, (b^2+b)
\end{equation}
has roots $x_1 \ = \ a$ and $x_2 \ = \ a'$, where $a'$ denotes the other root.
By Vieta’s formulas the roots satisfy
\begin{align} \label{eq:first}
a \, + \, a' \ = \ kb \, - \, 1
\end{align}
and
\begin{align} \label{eq:second}
a a'  \ = \ b^2 \, + \, b.
\end{align}
In particular, $a' \ = \ k b \, - \, 1 \, - \, a$ is an integer from \eqref{eq:first} and is positive from \eqref{eq:second}.
Since $a\,>\,b$, we have $a \, \geqslant \, b \, + \, 1$.
Therefore
\begin{equation} \label{a1LESSTHANa}
a'
\ = \
\frac{b^{2} + b}{a}
\, \leqslant \,
\frac{b^{2}+b}{b+1}
\ = \
b
\, < \,
a.
\end{equation}
Consequently, Vieta's formula produces a new integral solution $(a', \, b) \in S_{k}$ satisfying $a' \,+\, b \,<\, a \,+\, b$.

Hence $(a', \, b)$ is a solution with strictly smaller coordinate sum than $(a, \, b)$. Then the flipped pair $(b, \, a')$ is also a solution, satisfying $b \geqslant a'$ and having strictly smaller coordinate sum than $(a, \, b)$.

If $a' \ = \ b$, then $(b, \, a')$ is diagonal, hence equal to either $(1,\,1)$ or $(2,\,2)$, and therefore $k \ = \ 4$ or $3$, respectively.

Further, by Vieta jumps and flips, any non-diagonal solution $(a, \, b)$ jumps to a diagonal solution in finitely many steps.
If not, we would have infinitely many solutions with positive entries and strictly decreasing coordinate sums. This is impossible, since there are only finitely many pairs of positive integers whose coordinate sum is between $2$ and $a\,+\,b$.
The only diagonal solutions are $(2, \, 2)$ for $k \ = \ 3$ and $(1, \, 1)$ for $k \ = \ 4$, both of which have minimal coordinate sum for their corresponding value of $k$ by Lemma~\ref{lemma: minimalcoordinatesum}.
Therefore, every solution arises from a sequence generated by a diagonal solution, and the only such points are $(1, \, 1)$ and $(2, \, 2)$. Consequently, $k$ can only be $3$ and $4$.
\end{proof}

Theorem~\ref{Lemma: smallerCOORDINATEsum} shows the following.
\begin{enumerate}
\item $S_k$ is non-empty only for $k \ = \ 3$ and $k \ = \ 4$.
\item Every non-diagonal solution generates a chain of solutions ending in a diagonal solution.
\end{enumerate}

We now describe the reverse process of Vieta jumps and flips.

\begin{remark}
Theorem~$\ref{Lemma: smallerCOORDINATEsum}$ describes the reduction step in the process of Vieta jumps and flips. Starting from a non-diagonal solution $(a, \, b)$ with $a\,>\,b$, we apply Vieta’s formulas to the larger coordinate $a$ while keeping the smaller coordinate $b$ fixed. This produces a new solution $(a', \, b)$ in which the larger coordinate $a$ is replaced by a smaller one $a'$, thereby strictly decreasing the coordinate sum.

Reversing this process allows us to generate all solutions in $S_k$ starting from a diagonal solution. More precisely, let $(a,b)\in S_k$ with $a\geqslant b$. Viewing~\eqref{eq:quad_in_a} as a quadratic equation in the variable $b$, we obtain
\begin{equation}
b^2 \,-\, (ka-1)b \,+\, (a^2+a) \ = \ 0.
\end{equation}
The two roots of the quadratic polynomial $x^{2} \,-\, (ka-1) x \,+\, (a^{2}+a)$ are $b$ and $b'$, where $b'$ denotes the other root.
Then Vieta’s formulas give $b\,+\,b' \ = \ ka\,-\,1$ and $b b' \ = \ a^2 \, + \, a$.
Hence $b'$ is a positive integer and
\begin{equation}
b'
\ = \
\frac{a^2+a}{b}
\geqslant
\frac{b^2+b}{b}
\ = \
b\,+\,1,
\end{equation}
since $a\geqslant b$. Therefore, $(a, \, b')\in S_k$ with $b'\,>\,b$ and $a\,+\,b'\,>\,a\,+\,b$.

Thus, while the reduction process decreases the coordinate sum by replacing the larger coordinate with a smaller one, reversing the jump increases the coordinate sum and generates all non-diagonal solutions from a diagonal solution.
\end{remark}

\begin{example}
For $k \ = \ 3$, consider the solution $(a,\,b) \ = \ (6,\,3)$. Applying Vieta jumping to the larger coordinate $a \ = \ 6$ gives $a' \ = \ 3\cdot 3 \,-\, 1 \,-\, 6 \ = \ 2$, so we obtain the solution $(a', \, b) \ = \ (2,\,3)$, which has smaller coordinate sum.
Flipping and then applying Vieta jumping once more to the larger coordinate at $(\tilde{a}, \, \tilde{b}) \ = \ (3, \, 2)$ yields $\tilde{a}' \ = \ 3\cdot 2 \,-\, 1 \,-\, 3 \ = \ 2$, so we obtain the diagonal solution $(2, \, 2)$, which has the minimal coordinate sum in this chain.

On the other hand, applying Vieta jumping to the smaller coordinate $b \ = \ 3$ gives $b' \ = \ 3\cdot 6 \,-\, 1 \,-\,3 \ = \ 14$, yielding the solution $(a, \, b') \ = \ (6,\,14)$,
which has larger coordinate sum.
\end{example}

Therefore, every non-diagonal solution is obtained from a diagonal solution through a finite sequence of Vieta jumps and flips.

\medskip

We now prove parts~\ref{propert(i)} and~\ref{propert(ii)} of Theorem~\ref{thm:main} by analyzing separately the cases $k \ = \ 3$ and $k \ = \ 4$, corresponding to Theorems~\ref{thmS3} and~\ref{thmS4}, respectively.
The method in both cases is the same. Starting from a diagonal solution, we apply Vieta’s formulas to the quadratic polynomial~\eqref{eq:quad} to obtain the other solution.
We describe the solution set $S_k$ in each case.

We first describe $S_3$.
We study positive integer solutions $a,\,b$ of the Diophantine equation
\begin{equation}
\frac{a+1}{b} \;+\; \frac{b+1}{a} \ = \ 3.
\end{equation}
Starting from the solution with minimal coordinate sum $(a_0, \, a_1) \ := \ (2, \, 2)$, we recursively construct a sequence $\{a_n\}_{n\geqslant 0}$ of positive integers using Vieta’s formulas. For each $n \, \geqslant \, 2$, we define $a_n$ to be the other root of
\begin{equation}
x^2 \,-\, (3 a_{n-1} - 1)x \,+\, (a_{n-1}^2 + a_{n-1}) \ = \ 0.
\end{equation}
Vieta's formulas give
\begin{equation} \label{eq:closed3}
a_{n} \;+\; a_{n-2} \ = \ 3a_{n-1} - 1 \, , \qquad
a_{n} \, a_{n-2} \ = \ a_{n-1}^{2} + a_{n-1},
\quad (n \geqslant 2).
\end{equation}
Let $a_{n} \ = \ c_{n} \,+\, \delta$.
Substituting into the recurrence gives
\begin{equation}
(c_{n}+\delta) \,+\, (c_{n-2}+\delta) \ = \ 3 (c_{n}+\delta) \,-\, 1,
\end{equation}
so $c_{n} \,+\, c_{n-2} \ = \ 3 c_{n-1} \,+\, (3 \delta - 2 \delta - 1)$.
\\
Choosing $\delta \ = \ 1$ yields the homogeneous recurrence $c_{n} \,+\, c_{n-2} \ = \ 3 c_{n-1}$.
\\
The associated characteristic polynomial is $x^{2} \, - \, 3 x \,+\, 1$ whose roots are
\begin{equation}
\alpha \ = \ \frac{3 + \sqrt{5}}{2},
\qquad
\beta \ = \ \frac{3 - \sqrt{5}}{2}.
\end{equation}
Hence, by Binet's formula (see~\cite[Binet-like Formulas]{Koshy2019}), the general solution is $c_{n} \ = \ A \alpha^{n} \,+\, B \beta^{n}$.
\\
Using the initial conditions $c_{0} \ = \ c_{1} \ = \ 1$, we obtain
$\displaystyle{
\begin{cases}
A\,+\,B \ = \ c_{0}
\\
A\alpha \,+\, B\beta \ = \ c_{1}
\end{cases}
}$.
Solving this system gives
\begin{equation}
A \ = \ \frac{c_{1} - c_{0} \beta}{\alpha - \beta} \ = \ \frac{1-\beta}{\alpha-\beta} \ = \ \frac{-1+\sqrt{5}}{2\sqrt{5}} ,
\qquad
B \ = \ c_{0} \,-\, A \ = \ 1 \,-\, A \ = \ \frac{1+\sqrt{5}}{2\sqrt{5}}.
\end{equation}
Therefore, for all $n \geqslant 0$,
\begin{equation}
a_{n}
\ = \
\frac{-1 + \sqrt{5}}{2 \sqrt{5}} \left( \frac{3 + \sqrt{5
}}{2} \right)^{n} \;+\; \frac{1 + \sqrt{5}}{2 \sqrt{5}} \left( \frac{3 - \sqrt{5}}{2} \right)^{n} \;+\; 1.
\end{equation}
The sequence $\{a_n\}_{n \geqslant 0}$ corresponds to A032908 in the OEIS~\cite{OEIS}, and the first few terms are
\begin{align}
\begin{array}{c|c|c|c|c|c|c|c|c|c|c}
a_{0} & a_1 & a_{2} & a_{3} & a_{4} & a_{5} & a_{6} & a_{7} & a_{8} & a_{9} & a_{10} \\
\hline
2 & 2 & 3 & 6 & 14 & 35 & 90 & 234 & 611 & 1598 & 4182
\end{array}
\end{align}
Hence we have shown that for $k \ = \ 3$, every solution $(a, \, b) \in S_{3}$ with $a \, \leqslant \, b$ is of the form $a \ = \ a_n$ and $b \ = \ a_{n+1}$ for some $n \, \geqslant \, 0$.

\begin{lemma} \label{claim2}
Let $(a, \, b) \in S_{3}$ with $a \,\leqslant\, b$.
Then there is an $n \geqslant 0$ such that
\begin{equation}
a \ = \ F_{2n-1} \,+\, 1, \qquad b \ = \ F_{2n+1} \,+\, 1,
\end{equation}
where $\{F_n\}_{n \geqslant -1}$ is the Fibonacci sequence.
\end{lemma}

\begin{proof}
First of all, there exists $n \geqslant 0$ such that $a \ = \ a_{n}$ and $b \ = \ a_{n+1}$.
Using Binet's formula (see~\cite[Binet-like Formulas]{Koshy2019})
\begin{equation}
F_n \ = \ \frac{\varphi^n - (-\varphi)^{-n}}{\sqrt5},
\qquad
\varphi \ = \ \frac{1+\sqrt5}{2},
\end{equation}
a direct computation shows that
\begin{align*}
& F_{2n-1}\,+\,1
\ = \
\frac{\varphi^{-1}}{\sqrt{5}} (\varphi^{2})^{n} \,+\, \frac{\varphi}{\sqrt{5}} (\varphi^{-2})^{n} \,+\, 1
\\
\\
& \ = \ \frac{1}{\sqrt{5}} \cdot \frac{-1+\sqrt5}{2}\left(\frac{3+\sqrt5}{2}\right)^n
\,+\,
\frac{1}{\sqrt{5}} \cdot \frac{1+\sqrt5}{2}\left(\frac{3-\sqrt5}{2}\right)^n
\,+\,1
\\
\\
& \ = \ \frac{-1+\sqrt5}{2\sqrt5}\left(\frac{3+\sqrt5}{2}\right)^n
\,+\,
\frac{1+\sqrt5}{2\sqrt5}\left(\frac{3-\sqrt5}{2}\right)^n
\,+\,1.
\end{align*}
Comparing this expression with the closed form for $a_n$ proves that
\begin{equation} \label{an:definition}
a_n \ = \ F_{2n-1}\,+\,1,
\end{equation}
thus $a \ = \ F_{2n-1}\,+\,1$ and $b \ = \ F_{2n+1}\,+\,1$.
\end{proof}

The following theorem establishes part~\ref{propert(i)} of Theorem~\ref{thm:main}. We now summarize the preceding discussion by giving a complete description of all solutions in $S_3$.

\begin{theorem}[Part~\ref{propert(i)} of Theorem~\ref{thm:main}] \label{thmS3}
The set of all solutions in $S_{3}$ is given by
\begin{equation}
S_{3}
\ = \
\big \{(F_{2n-1}+1, \, F_{2n+1}+1) : n\geqslant 0 \big \}
\cup
\big \{ (F_{2n+1}+1, \, F_{2n-1}+1) : n\geqslant 1 \big \}.
\end{equation}
In particular,
\begin{equation}
(2, \, 2), \,
(3, \, 2), \,
(6, \, 3), \,
(14, \, 6), \,
(35, \, 14), \,
(90, \, 35), \,
(234, \, 90), \,
(611, \, 234), \,
\ldots \in S_{3}.
\end{equation}
\end{theorem}

\medskip

In the same way as for $S_3$, we now describe $S_4$.
We study positive integer solutions $a, \, b$ of the Diophantine equation
\begin{equation}
\frac{a+1}{b} \;+\; \frac{b+1}{a} \ = \ 4.
\end{equation}
Starting from the diagonal solution with minimal coordinate sum $(b_0, \, b_{1}) \ := \ (1, \, 1)$, we recursively construct a sequence $\{b_n\}_{n\geqslant 0}$ of positive integers using Vieta’s formulas.
For each $n \geqslant 2$, we define $b_n$ to be the other root of
\begin{equation}
x^2 \,-\, (4b_{n-1} - 1)x \,+\, (b_{n-1}^2 + b_{n-1}) \ = \ 0.
\end{equation}
By Vieta's formulas,
\begin{equation}
b_{n} \, + \, b_{n-2} \ = \ 4 b_{n-1} \,-\, 1,
\qquad
b_{n} b_{n-2} \ = \ b_{n-1}^{2} \, + \, b_{n-1}
\qquad (n \geqslant 2).
\end{equation}
Let $b_{n} \ = \ c_{n} \,+\, \delta$.
Substituting into the recurrence gives
\begin{equation}
(c_{n}+\delta) \,+\, (c_{n-2}+\delta) \ = \ 4 (c_{n}+\delta) \,-\, 1,
\end{equation}
so $c_{n} \,+\, c_{n-2} \ = \ 4 c_{n-1} \,+\, (4 \delta - 2 \delta - 1)$.
\\
Choosing $\delta \ = \ 1/2$ yields the homogeneous recurrence $c_{n} \,+\, c_{n-2} \ = \ 4 c_{n-1}$.
\\
The associated characteristic polynomial is $x^{2} \, - \, 4 x \,+\, 1$ whose roots are
\begin{equation}
\alpha \ = \ \frac{4 + \sqrt{12}}{2} \ = \ 2 \,+\, \sqrt{3},
\qquad
\beta \ = \ \frac{4 - \sqrt{12}}{2} \ = \ 2 \,-\, \sqrt{3}.
\end{equation}
Hence, by Binet's formula (see~\cite[Binet-like Formulas]{Koshy2019}), the general solution is $c_{n} \ = \ A \alpha^{n} \,+\, B \beta^{n}$.
\\
Using the initial conditions $c_{0} \ = \ c_{1} \ = \ 1/2$, we obtain
$\displaystyle{
\begin{cases}
A\,+\,B \ = \ c_{0}
\\
A\alpha \,+\, B\beta \ = \ c_{1}
\end{cases}
}$.
Solving this system gives
\begin{equation}
A \ = \ \frac{c_{1} - c_{0} \beta}{\alpha - \beta} \ = \ \frac{1}{2} \cdot \frac{1-\beta}{\alpha-\beta} \ = \ \frac{1}{2} \cdot \frac{-1+\sqrt{3}}{2\sqrt{3}} \ = \ \frac{3-\sqrt{3}}{12}
\end{equation}
and
\begin{equation}
B \ = \ c_{0} \,-\, A \ = \ \frac{1}{2} \,-\, A \ = \ \frac{3+\sqrt{3}}{12}.
\end{equation}
Therefore, for all $n \geqslant 0$,
\begin{equation}\label{eq:bn-closed}
b_n
\ = \
\frac{(3-\sqrt{3})}{12} (2 \,+\, \sqrt{3})^{n} \;+\; \frac{(3+\sqrt{3})}{12} (2 \,-\, \sqrt{3})^{n} \;+\; \frac{1}{2}.
\end{equation}
The sequence $\{b_n\}_{n \geqslant 0}$ corresponds to A101879 in the OEIS~\cite{OEIS}, and the first few terms are
\begin{equation}
\begin{array}{c|c|c|c|c|c|c|c|c|c|c}
b_{0} & b_{1} & b_{2} & b_{3} & b_{4} & b_{5} & b_{6} & b_{7} & b_{8} & b_{9} & b_{10} \\
\hline
1 & 1 & 2 & 6 & 21 & 77 & 286 & 1066 & 3977 & 14841 & 55386
\end{array}
\end{equation}
Hence we have shown that for $k \ = \ 4$, every solution $(a, \, b) \in S_{4}$ with $a \, \leqslant \, b$ is of the form $a \ = \ b_n$ and $b \ = \ b_{n+1}$ for some $n \, \geqslant \, 0$.

The following theorem establishes part~\ref{propert(ii)} of Theorem~\ref{thm:main}. We now summarize the preceding discussion by giving a complete description of all solutions in $S_4$.

\begin{theorem}[Part~\ref{propert(ii)} of Theorem~\ref{thm:main}] \label{thmS4}
The set of all solutions in $S_{4}$ is given by
\begin{equation}
S_{4}
\ = \
\big \{(b_{n}, \, b_{n+1}) : n\geqslant 0 \big \}
\cup
\big \{ (b_{n+1}, \, b_{n}) : n\geqslant 1 \big \}.
\end{equation}
In particular,
\begin{equation}
(1, \, 1), \, (2, \, 1), \, (6, \, 2), \, (21, \, 6), \, (77, \, 21), \, (286, \, 77), \, \ldots \in S_{4}.
\end{equation}
\end{theorem}

\section{Proof of Theorem \ref{cor:normalized_values}} \label{section: lasttheorem}

We break the proof into two separate cases: $k \ = \ 3$ and $k \ = \ 4$, corresponding to Theorems~\ref{thm:k3} and~\ref{thm:k4}, respectively, which together establish Theorem~\ref{cor:normalized_values}.
\subsection{Analysis for $k \ = \ 3$} \label{CORk=3} \label{CORk=3}

We express the sequence $\{a_n\}_{n \geqslant 0}$ defined by~\eqref{an:definition} in terms of the Fibonacci numbers $\{F_n\}_{n \geqslant -1}$ and Lucas numbers $\{L_n\}_{n \geqslant 0}$ to study arithmetic properties such as $\gcd(a_n, \, a_{n+1})$.
Recall that Lucas numbers are defined by
\begin{equation}
L_{0} \ = \ 2,
\qquad
L_{1} \ = \ 1,
\qquad
L_{n} \ = \ L_{n-1}\,+\,L_{n-2}
\quad \text{for all } n \geqslant 2.
\end{equation}
Using the identities (see \cite[Chapter 31]{Koshy2019})
\begin{equation}
F_{2n-1} \ = \ F_n^2 \,+\, F_{n-1}^2,
\quad
F_{n+1}F_{n-1}\,-\,F_n^2 \ = \ (-1)^n,
\quad
L_n \ = \ F_{n-1}\,+\,F_{n+1},
\end{equation}
we obtain
\begin{equation}
F_{4n-1} \ = \ F_{2n-1}L_{2n} \,-\, 1,
\qquad
F_{4n-3} \ = \ F_{2n-1}L_{2n-2} \,-\, 1.
\end{equation}
Hence
\begin{equation} \label{formulass}
a_{2n} \ = \ F_{4n-1}+1 \ = \ F_{2n-1}L_{2n},
\qquad
a_{2n-1} \ = \ F_{4n-3}+1 \ = \ F_{2n-1}L_{2n-2}.
\end{equation}

\begin{lemma} \label{claim3}
Let $\{a_n\}_{n \geqslant 0}$ be defined by~\eqref{an:definition} and $d_{n} \ := \ \gcd(a_{n}, \, a_{n+1})$ for all $n \, \geqslant \, 1$.
Then,
\begin{equation}
d_{2n-1} \ = \ F_{2n-1}
\qquad\text{and}\qquad
d_{2n} \ = \ L_{2n}
\qquad
\text{for all } n \geqslant 1.
\end{equation}
\end{lemma}

\begin{proof}
We use the standard facts (see \cite[Chapter 32]{Koshy2019})
\begin{equation*}
\gcd(F_{2n-1},\,F_{2n}) \ = \ 1
\quad\text{and}\quad
\gcd(L_{2n-2},\,L_{2n-1}) \ = \ 1
\qquad \text{for all } n \geqslant 1.
\end{equation*}
We consider two cases.
\begin{itemize}
\setlength\itemsep{0.8em}

\item Odd--indexed of $d_{n}$: Using the formulas~\eqref{formulass} for $a_{2n}$ and $a_{2n-1}$, we obtain
\begin{align*}
d_{2n-1}
&\ = \
\gcd(a_{2n-1}, \, a_{2n})
\\
\\
&\ = \
\gcd(F_{2n-1}L_{2n-2},\,F_{2n-1}L_{2n})
\\
\\
&\ = \
F_{2n-1} \gcd(L_{2n-2}, \, L_{2n})
\\
\\
&\ = \
F_{2n-1} \gcd(L_{2n-2},\,L_{2n-1} \,+\, L_{2n-2})
\\
\\
&\ = \
F_{2n-1} \gcd(L_{2n-2},\,L_{2n-1})
\\
\\
&\ = \
F_{2n-1}.
\end{align*}

\item Even--indexed of $d_{n}$: Similarly,
\begin{align*}
d_{2n}
& \ = \
\gcd(a_{2n},\,a_{2n+1})
\\
\\
& \ = \
\gcd(F_{2n-1}L_{2n},\,F_{2n+1}L_{2n})
\\
\\
& \ = \
L_{2n} \gcd(F_{2n-1},\,F_{2n+1})
\\
\\
& \ = \
L_{2n} \gcd(F_{2n-1},\,F_{2n} \,+\, F_{2n-1})
\\
\\
& \ = \
L_{2n} \gcd(F_{2n-1},\,F_{2n})
\\
\\
& \ = \
L_{2n}. \qedhere
\end{align*}
\end{itemize}
\end{proof}

\begin{lemma} \label{claim4}
Let $\{a_n\}_{n \geqslant 0}$ be defined by~\eqref{an:definition} and $d_{n} \ := \ \gcd(a_{n}, \, a_{n+1})$ for all $n \, \geqslant \, 1$.
Then, $a_{n+1} \ =\ d_nd_{n+1}$ for all $n \, \geqslant \, 1$.
\end{lemma}

\begin{proof}
By Lemma~\ref{claim3}, we have $d_{2n-1} \ = \ F_{2n-1}$ and $d_{2n} \ = \ L_{2n}$.
Hence
\begin{equation*}
d_{2n-1}d_{2n}
\ = \
F_{2n-1}L_{2n}
\ = \
a_{2n},
\end{equation*}
and
\begin{equation*}
d_{2n-2}d_{2n-1}
\ = \
L_{2n-2}F_{2n-1}
\ = \
a_{2n-1}.
\end{equation*}
Therefore, $a_{n+1} \ = \ d_n d_{n+1}$ for all $n\geqslant 1$.
\end{proof}

We now prove Theorem~\ref{cor:normalized_values} in the case $k \ = \ 3$.

\begin{theorem}\label{thm:k3}
Let $\{a_n\}_{n \geqslant 0}$ be defined by~\eqref{an:definition}.
Then, for all $n \geqslant 1$, we have
\begin{equation*}
\frac{a_{2n}\,+\,a_{2n+1}}{\gcd(a_{2n},\,a_{2n+1})^2} \ = \ 1,
\qquad
\frac{a_{2n-1}\,+\,a_{2n}}{\gcd(a_{2n-1},\,a_{2n})^2} \ = \ 5.
\end{equation*}
\end{theorem}

\begin{proof} \label{proofk=3}
By Lemmas~\ref{claim3} and~\ref{claim4}, we have $d_{2n-1} \ = \ F_{2n-1}$, $d_{2n} \ = \ L_{2n}$ and $a_{n+1} \ = \ d_{n} d_{n+1}$.
Then
\begin{equation} \label{STAR}
\frac{a_{n-1}\,+\,a_n}{\gcd(a_{n-1},\,a_n)^2}
\ = \
\frac{ d_{n-1} d_{n-2} \,+\, d_{n} d_{n-1} }{ d_{n-1}^{2} }
\ = \
\frac{ d_{n-2} \,+\, d_{n} }{ d_{n-1} }.
\end{equation}
We proceed by considering even and odd indices.
Upon replacing $n$ with $2n+1$ in~\eqref{STAR}, we obtain
\begin{equation*}
\frac{a_{2n}\,+\,a_{2n+1}}{\gcd(a_{2n},\,a_{2n+1})^2}
\ = \  \frac{d_{2n-1}\,+\,d_{2n+1}}{d_{2n}}
\ = \  \frac{F_{2n-1}\,+\,F_{2n+1}}{L_{2n}} \ = \  1.
\end{equation*}
Upon replacing $n$ with $2n$ in~\eqref{STAR}, we obtain
\begin{align*}
& \frac{a_{2n-1}\,+\,a_{2n}}{\gcd(a_{2n-1},\,a_{2n})^2}
\ = \ \frac{d_{2n-2}\,+\,d_{2n}}{d_{2n-1}}
\ = \ \frac{L_{2n-2}\,+\,L_{2n}}{F_{2n-1}}
\\
\\
& \ = \ \frac{F_{2n-3} \,+\, F_{2n-1} \,+\, F_{2n-1} \,+\, F_{2n+1}}{F_{2n-1}}
\\
\\
& \ = \ \frac{F_{2n-3} \,+\, 2 F_{2n-1} \,+\, ( F_{2n} \,+\, F_{2n-1} )}{F_{2n-1}}
\\
\\
& \ = \ \frac{F_{2n-3} \,+\, 3 F_{2n-1} \,+\, ( F_{2n-1} \,+\, F_{2n-2} )}{F_{2n-1}}
\\
\\
& \ = \ \frac{( F_{2n-3} \,+\, F_{2n-2} ) \,+\, 4 F_{2n-1} }{F_{2n-1}} \ = \
\frac{F_{2n-1} \,+\, 4 F_{2n-1} }{F_{2n-1}} \ = \  \frac{5F_{2n-1}}{F_{2n-1}} \ = \  5. \qedhere
\end{align*}
\end{proof}

\subsection{Analysis for $k \ = \ 4$} \label{CORk=4}

The sequence $\{b_n\}_{n\geqslant 0}$ defined by~\eqref{eq:bn-closed}, arising from the iteration of Vieta’s formulas, appears in the OEIS~\cite{OEIS} as sequence A101879.
The sequence $d_n \ := \ \gcd(b_n, \, b_{n+1}), \, n \geqslant 0$ is listed as A005246.
Its odd--indexed and even--indexed subsequences are OEIS sequences A001075 and A001835, respectively.
Both the identity
\begin{equation}\label{eq:bn-dn}
b_{n+1} \ = \ d_{n+1} d_n \qquad \text{for all } n \; \geqslant \; 0
\end{equation}
and the closed--form expressions for $d_{2n-1}$ and $d_{2n}$,
\begin{equation} \label{d2n-1}
d_{2n-1}
\ = \
\frac{3-\sqrt{3}}{6} \alpha^{n}
\;+\;
\frac{3+\sqrt{3}}{6} \beta^{n} \qquad \text{for all } n \; \geqslant \; 1
\end{equation}
and
\begin{equation} \label{d2n}
d_{2n}
\ = \
\frac{\alpha^{n} \;+\; \beta^{n}}{2} \qquad \text{for all } n \; \geqslant \; 1
\end{equation}
are recorded in OEIS entry A005246 and its references.

We now prove Theorem~\ref{cor:normalized_values} in the case $k \ = \ 4$.

\begin{theorem} \label{thm:k4}
Let $\{b_n\}_{n\geqslant 0}$ be defined by~\eqref{eq:bn-closed}.
Then, for all $n \geqslant 1$, we have
\begin{equation*}
\frac{b_{2n}\,+\,b_{2n+1}}{\gcd(b_{2n},\,b_{2n+1})^2} \ = \ 2,
\qquad
\frac{b_{2n-1}\,+\,b_{2n}}{\gcd(b_{2n-1},\,b_{2n})^2} \ = \ 3.
\end{equation*}
\end{theorem}

\begin{proof}
From~\eqref{eq:bn-dn}, \eqref{d2n-1}, and~\eqref{d2n}, we obtain
\begin{equation} \label{STAR2}
\frac{ b_{n-1} \,+\, b_{n} }{ \gcd(b_{n-1}, \, b_{n})^{2} } \ = \ \frac{ d_{n-1} d_{n-2} \,+\, d_{n} d_{n-1} }{ d_{n-1}^{2} } \ = \ \frac{ d_{n-2} \,+\, d_{n} }{ d_{n-1} }.
\end{equation}
We now proceed by considering even and odd indices separately.
Upon replacing $n$ with $2n+1$ in~\eqref{STAR2}, we obtain
\begin{align*}
\frac{ b_{2n} \,+\, b_{2n+1} }{ \gcd(b_{2n},\, b_{2n+1})^{2} }
\ = \
\frac{ d_{2n-1} \,+\, d_{2n+1} }{ d_{2n} }
\ = \
\frac{ \left( \frac{3 - \sqrt{3}}{6} \alpha^{n} + \frac{3 + \sqrt{3}}{6} \beta^{n} \right) + \left( \frac{3 - \sqrt{3}}{6} \alpha^{n+1} + \frac{3 + \sqrt{3}}{6} \beta^{n+1} \right) }{ \frac{1}{2} \alpha^{n} + \frac{1}{2} \beta^{n} }
\\
\\
\ = \ 2 \frac{ \left( \frac{3 - \sqrt{3}}{6} \alpha^{n} + \frac{3 + \sqrt{3}}{6} \beta^{n} \right) + \left( \frac{3 + \sqrt{3}}{6} \alpha^{n} + \frac{3 - \sqrt{3}}{6} \beta^{n} \right) }{ \alpha^{n} + \beta^{n} }
\ = \
2 \frac{ \alpha^{n} + \beta^{n} }{ \alpha^{n} + \beta^{n} }
\ = \ 2.
\end{align*}
Upon replacing $n$ with $2n$ in~\eqref{STAR2}, we obtain
\begin{align*}
& \frac{ b_{2n-1} \,+\, b_{2n} }{ \gcd{( b_{2n-1} , \, b_{2n} )}^{2} } \, = \, \frac{d_{2n-2} \,+\, d_{2n}}{d_{2n-1}}
\ = \
\frac{ \left( \frac{1}{2} \alpha^{n-1} + \frac{1}{2} \beta^{n-1} \right) + \left( \frac{1}{2} \alpha^{n} + \frac{1}{2} \beta^{n} \right) }{ \frac{3 - \sqrt{3}}{6} \alpha^{n} + \frac{3 + \sqrt{3}}{6} \beta^{n} }
\\
\\
&
\ = \ \frac{ \left( \frac{1}{2} \alpha^{n-1} + \frac{1}{2} \beta^{n-1} \right) + \left( \frac{2+\sqrt{3}}{2} \alpha^{n-1} + \frac{2-\sqrt{3}}{2} \beta^{n-1} \right) }{ \frac{3 + \sqrt{3}}{6} \alpha^{n-1} + \frac{3 - \sqrt{3}}{6} \beta^{n-1} }
\\
\\
&
\ = \ \frac{ \frac{3 + \sqrt{3}}{2} \alpha^{n-1} + \frac{3 - \sqrt{3}}{2} \beta^{n-1} }{ \frac{3 + \sqrt{3}}{6} \alpha^{n-1} + \frac{3 - \sqrt{3}}{6} \beta^{n-1} } \ = \ \frac{1/2}{1/6} \ = \ 3.\qedhere
\end{align*}
\end{proof}

To summarize Theorems~\ref{thm:k3} and~\ref{thm:k4}, we collect the results in the following theorem.

\begin{theorem}[Theorem~\ref{cor:normalized_values}]
Let $(a,\,b)$ be a positive integer solution of
\begin{equation*}
\frac{a+1}{b}\,+\,\frac{b+1}{a}\ = \ k.
\end{equation*}
Then
\begin{equation*}
\frac{a+b}{\gcd(a,b)^2}
\,\in\,\{1,\,2,\,3,\,5\}.
\end{equation*}
Moreover,
\begin{itemize}
\item if $k \ = \ 3$, the values are $1$ and $5$;
\item if $k \ = \ 4$, the values are $2$ and $3$.
\end{itemize}
\end{theorem}

\section{Limitations of Vieta Jumping and Open Problems}
\label{sec:limitations_open}

\subsection{Difficulties for general $r$}

Let $(a, \, b)$ be a positive integer solution of
\begin{equation} \label{general-r}
\frac{a+r}{b}\,+\,\frac{b+r}{a} \ = \ k, \qquad r \in \mathbb{Z}_{\geqslant 2}.
\end{equation}

We first determine the diagonal solutions of~\eqref{general-r}.
If $a \ = \ b$, then~\eqref{general-r}
becomes $2\,+\,2r/a \ = \ k$.
Hence there are no solutions when $k \ = \ 1$ or $k \ = \ 2$, since the left-hand side is strictly greater than $2$.
Solving for $a$, we obtain $a \ = \ \tfrac{2r}{k-2}$.
Thus diagonal solutions exist precisely when $(k\,-\,2)$ divides $2r$.

Now suppose $(a, \, b)$ is a non-diagonal solution with $a>b$. Then $a \, \geqslant \, b \, + \, 1$.
Consider the quadratic equation $x^2 \,-\, (kb-r)x \,+\, (b^2+rb) \ = \ 0$. Since $x \ = \ a$ is a root, let $a'$ denote the other root.
By Vieta's formulas, $a \,+\, a' \ = \ kb \,-\, r$ and $aa' \ = \ b^2 \,+\, rb$.
Thus $a' \ = \ kb \,-\, r \,-\, a$ is again a positive integer.
Using $a \, \geqslant \, b \,+\, 1$, we obtain
\begin{equation} \label{a'b+r}
a'
\ = \
\frac{b^2+rb}{a}
\ = \
b \frac{b + 1}{a} \;+\; (r-1) \frac{b}{a}
\; \leqslant \;
b \;+\; (r-1) \frac{b}{a}
\; < \;
b \;+\; (r\,-\,1).
\end{equation}
The problem is that we no longer have a proof that $a'\,<\,a$ once $r \, \geqslant \, 3$, so the descent argument breaks down.

\medskip

The following remark shows that the descent argument of Theorem~\ref{Lemma: smallerCOORDINATEsum} still applies when $r \ = \ 2$.

\begin{remark}
Suppose $r \ = \ 2$.
Since $a \, > \, b$ and $a \, \geqslant \, b \,+\, 1$, then~\eqref{a'b+r} gives
\begin{equation}
a' \ = \ \frac{b^{2} + 2 b}{a}
\ = \ b \frac{b + 1}{a} \;+\; \frac{b}{a} \, < \,  b \;+\; 1 \, \leqslant \, a.
\end{equation}
Hence every non-diagonal solution $(a, \, b)$ produces a new solution $(a', \, b)$ with strictly smaller coordinate sum.

We now determine the diagonal solutions. If $(a, \, a)$ is a solution, then $a \ = \ \tfrac{4}{k-2}$. Thus $(k\,-\,2)$ divides $4$, so $k\,-\,2\,\in\,\{1,\,2,\,4\}$. Therefore $k \, \in \, \{ 3, \, 4, \, 6 \}$.
Hence the only diagonal solutions are $(1, \, 1)$ (corresponding to $k \ = \ 6$), $(2, \, 2)$ (corresponding to $k \ = \ 4$) and $(4, \, 4)$ (corresponding to $k \ = \ 3$).

Applying the same descent argument as in Theorem~$\ref{Lemma: smallerCOORDINATEsum}$, every non-diagonal solution generates a chain of solutions ending in a diagonal solution via repeated applications of the Vieta jumps
\begin{equation}
(a, \, b) \longmapsto (kb \,-\,2\,-\,a , \, b) \qquad \text{or} \qquad (a, \, b) \longmapsto (a , \, ka \,-\,2\,-\,b)
\end{equation}
together with the flip
\begin{equation}
(a, \, b) \longmapsto (b , \, a).
\end{equation}

Consequently, when $r \ = \ 2$, the only possible values of $k$ are $3$, $4$ and $6$.
\end{remark}

\begin{conjecture}
Let $(a,\,b)$ be a positive integer solution of
\begin{equation*}
\frac{a+2}{b}\,+\,\frac{b+2}{a}\ = \ k.
\end{equation*}
Then
\begin{equation*}
\frac{2(a+b)}{\gcd(a,b)^2}
\,\in\,\left\{ 1 , \; 2 , \; 3 , \; 4 , \; 5  , \; 8 \right\}.
\end{equation*}
Moreover,
\begin{itemize}\itemsep0.5em
\item if $k \ = \ 3$, the values are $1$ and $5$;
\item if $k \ = \ 4$, the values are $2$ and $3$;
\item if $k \ = \ 6$, the values are $4$ and $8$.
\end{itemize}
\end{conjecture}

\subsection{Failure in the three--variable case}

Consider the generalization to
\begin{equation}\label{eq:m3}
\frac{a+r}{b}\;+\;\frac{b+r}{c}\;+\;\frac{c+r}{a} \ = \ k,
\qquad a, \, b, \, c, \, r, \, k \in \mathbb{Z}_{>0}.
\end{equation}
Fixing $b, \, c$ yields a quadratic in $a$ with root $a'$ satisfying
\begin{equation}\label{eq:vieta3}
a' \ = \ b\!\left(k-\frac{b+r}{c}\right)\,-\,r\,-\,a,
\qquad
aa' \ = \ b(c+r),
\end{equation}
and $a'$ need not be integral. Thus the recursive structure of the two--variable case can break down.

\begin{example}
Let $r \ = \ 4$, $k \ = \ 7$. Applying Vieta jumping to the solution $(66, \, 48, \, 352)$ does not preserve integrality, as it produces non-integer triples such as
\begin{equation*}
\left( \frac{2848}{11}, \, 48, \, 352 \right) ,
\qquad
\left( 66, \, \frac{39}{4}, \, 352 \right) ,
\qquad
\left( 66, \, 48, \, \frac{1540}{3} \right).
\end{equation*}
This motivates the study of rational solutions.
\end{example}

\begin{example}
Let $r \ = \ 2$, $k \ = \ 7$. No symmetric positive integer solution $(s, \, s, \, s)$ exists. Among integer solutions, $(7, \, 18, \, 40)$ has minimal sum (verified computationally using Python code).
\end{example}

These examples show that Vieta jumping may fail to preserve integrality, symmetric minima may not exist, and solutions can split into disjoint families not connected by Vieta jumps.

\subsection{Generalizations and open problems}

For $r, \, k\in\mathbb{Z}_{>0}$ and $m \, \geqslant \, 2$, instead of restricting to the two-- or three--variable cases, we consider the problem of finding positive integers
$a_{1}, \, a_{2}, \, \ldots , \, a_{m}$ satisfying
\begin{equation}
\frac{a_{1}+r}{a_{2}} \,+\, \frac{a_{2}+r}{a_{3}} \,+\, \cdots \,+\, \frac{a_{m}+r}{a_{1}} \ = \ k.
\end{equation}
While the case $r \ = \ 1$ and $m \ = \ 2$ is completely understood, the general problem for $r \, \geqslant \, 2$ and $m \, \geqslant \, 3$ remains open. Key questions include finiteness, effective bounds, and whether Vieta jumping generates new integral solutions.

A natural extension replaces $a_i\,+\,r$ by polynomials:
\begin{equation}\label{eq:poly_general}
\sum\limits_{i=1}^{m}\frac{f_i(a_i)}{a_{i+1}} \ = \ k,
\qquad f_i \in \mathbb{Z}[x].
\end{equation}
Another open case is the symmetric equation
\begin{equation}
\frac{a+r}{b}\;+\;\frac{a+r}{c}
\;+\;
\frac{b+r}{a}\;+\;\frac{b+r}{c}
\;+\;
\frac{c+r}{a}\;+\;\frac{c+r}{b}
\ = \ k,
\end{equation}
whose arithmetic structure remains poorly understood.

\medskip
\noindent MSC2020: 11B39, 33C05


\end{document}